\documentclass[preprint,11pt]{elsarticle}

\usepackage{amsfonts, amsmath, amscd}
\usepackage[psamsfonts]{amssymb}

\usepackage{amssymb}

\usepackage{pb-diagram}

\usepackage[all,cmtip]{xy}

\usepackage[usenames]{color}

\newtheorem{theorem}{Theorem}[section]
\newtheorem{lemma}[theorem]{Lemma}
\newtheorem{corollary}[theorem]{Corollary}

\newtheorem{remark}[theorem]{Remark}
\newtheorem{proposition}[theorem]{Proposition}
\newtheorem{definition}[theorem]{Definition}
\newtheorem{example}[theorem]{Example}

\newproof{proof}{Proof}

\numberwithin{equation}{section}
\numberwithin{theorem}{section}

\newcommand{\e}{\varepsilon}
\newcommand{\w}{\omega}

\newcommand{\ff}{\mathbb{F}}

\newcommand{\IF}{\mathbb{F}}

\newcommand{\Nn}{\mathcal{N}}

\newcommand{\LL}{\mathcal{L}}

\newcommand{\cl}{\mathrm{cl}}
\newcommand{\Ra}{\Rightarrow}

\newcommand{\cacx}{\overline{\mathrm{acx}}}

\newcommand{\Bo}{\mathsf{Bo}}

\newcommand{\Id}{\mathsf{id}}

\begin{document}

\begin{frontmatter}

\title{Weakly and weak$^\ast$ $p$-convergent operators}

\author{Saak Gabriyelyan}
\ead{saak@math.bgu.ac.il}
\address{Department of Mathematics, Ben-Gurion University of the Negev, Beer-Sheva, P.O. 653, Israel}

\begin{abstract}
Let $p\in[1,\infty]$. Being motivated by weakly $p$-convergent and weak$^\ast$ $p$-convergent operators between Banach spaces introduced by Fourie and Zeekoei, we introduce and study the classes of weakly $p$-convergent and weak$^\ast$ $p$-convergent operators between arbitrary locally convex spaces. Relationships between these classes of operators are given, and we show that they have ideal properties. Numerous characterizations of weakly $p$-convergent and weak$^\ast$ $p$-convergent operators are given.
\end{abstract}


\end{frontmatter}


\section{Introduction}


Unifying the notion of unconditionally converging operators and the notion of completely continuous operators, Castillo and S\'{a}nchez selected in \cite{CS} the class of $p$-convergent operators. An operator $T:X\to Y$ between Banach spaces $X$ and $Y$ is called {\em $p$-convergent} if it transforms weakly $p$-summable sequences into norm null sequences (all relevant definitions are given in Section \ref{sec:pre}). Using this notion they introduced and study Banach spaces with the Dunford--Pettis property of order $p$ ($DPP_p$ for short)  for every $p\in[1,\infty]$. A Banach space $X$ is said to have the $DPP_p$ if every weakly compact operator from $X$ into a Banach space $Y$ is $p$-convergent.

The influential article  of Castillo and S\'{a}nchez \cite{CS} inspired an intensive study of $p$-versions of numerous geometrical properties of Banach spaces and new classes  of  operators of $p$-convergent type. The following two classes of operators between Banach spaces were introduced and studied by Fourie and Zeekoei in \cite{FZ-con} and  \cite{FZ-p}, respectively, where the Banach dual of a Banach space $X$ is denoted by $X^\ast$.

\begin{definition} \label{def:weak-p-convergent-Banach} {\em
Let $p\in[1,\infty]$, and let $X$ and $Y$ be Banach spaces. An operator $T:X\to Y$ is called
\begin{enumerate}
\item[{\rm(i)}] {\em weakly $p$-convergent}  if
$
\lim_{n\to\infty} \langle\eta_n, T(x_n)\rangle=0
$
for every weakly null sequence $\{\eta_n\}_{n\in\w}$ in $Y^\ast$ and each weakly $p$-summable sequence $\{x_n\}_{n\in\w}$ in $X$;
\item[{\rm(ii)}] {\em weak$^\ast$ $p$-convergent} if
$
\lim_{n\to\infty} \langle\eta_n, T(x_n)\rangle=0
$
for every weak$^\ast$ null sequence $\{\eta_n\}_{n\in\w}$ in $Y^\ast$ and each weakly $p$-summable sequence $\{x_n\}_{n\in\w}$ in $X$.\qed
\end{enumerate}}
\end{definition}
It should be mentioned that if $p=\infty$, weakly $p$-convergent operators are known as {\em weak Dunford--Pettis} operators (see \cite[p.~349]{AB}) and  weak$^\ast$ $p$-convergent operators are known as {\em weak$^\ast$ Dunford--Pettis} operators (see \cite{EKHBM}).

Numerous characterizations and  applications of weakly $p$-convergent and weak$^\ast$ $p$-convergent operators between Banach spaces and in particular between Banach lattices were obtained in  \cite{FZ-con,FZ-p,FZ-pL}. These results motivate us to consider these classes of operators in the general case of locally convex spaces.

\begin{definition} \label{def:weak-p-convergent} {\em
Let $p\in[1,\infty]$, and let $E$ and $L$ be separated topological vector spaces. A linear map $T:E\to L$ is called
\begin{enumerate}
\item[{\rm(i)}] {\em weakly $p$-convergent}  if
$
\lim_{n\to\infty} \langle\eta_n, T(x_n)\rangle=0
$
for every weakly null sequence $\{\eta_n\}_{n\in\w}$ in $L'_\beta$ and each weakly $p$-summable sequence $\{x_n\}_{n\in\w}$ in $E$;
\item[{\rm(ii)}] {\em weak$^\ast$ $p$-convergent} if
$
\lim_{n\to\infty} \langle\eta_n, T(x_n)\rangle=0
$
for every weak$^\ast$ null sequence $\{\eta_n\}_{n\in\w}$ in $L'$ and each weakly $p$-summable sequence $\{x_n\}_{n\in\w}$ in $E$.
\end{enumerate}}
\end{definition}

Relationships between the classes of $p$-convergent, weakly $p$-convergent and  weak$^\ast$ $p$-convergent operators  are given in Proposition \ref{p:weak*-p-convergent-weak-con}, and  Proposition \ref{p:weak*-p-convergent=weak-con} provides a sufficient condition on the range space $L$ under which all these three classes of operators coincide.

In \cite{BourDies} Bourgain and Diestel introduced the class of limited operators between Banach spaces. More general classes of limited completely continuous and  limited $p$-convergent operators were defined and studied by Salimi and  Moshtaghioun \cite{SaMo}  and Fourie and Zeekoei  \cite{FZ-p}, respectively. We generalize these classes by introducing the classes of $(q',q)$-limited $p$-convergent and   $(q',q)$-$(V^\ast)$ $p$-convergent operators from a locally convex space $E$ to a locally convex space $L$, where $p,q,q'\in[1,\infty]$ and $q'\leq q$. In Proposition \ref{p:weak-pq-convergent-ideal} we show that these new classes have ideal properties.
If the space $L$ contains an isomorphic copy of $\ell_\infty$, in Theorems \ref{t:weak-p-conv-limited} and \ref{t:weak-p-conv-V*} we show that all weakly $p$-convergent  operators are $(q',q)$-limited $p$-convergent (resp.,  $(q',q)$-$(V^\ast)$ $p$-convergent) if and only if so are all   weak$^\ast$ $p$-convergent operators if and only if so is the identity operator $\Id_E:E\to E$.

The main results of the article are Theorems \ref{t:weak*-p-convergent} and \ref{t:char-weakly-p-convergent} in which we give numerous characterizations of weak$^\ast$ $p$-convergent and weakly $p$-convergent  operators between locally convex spaces.


\section{Preliminaries results} \label{sec:pre}


We start with some necessary definitions and notations used in the article. Set $\w:=\{ 0,1,2,\dots\}$. All topological spaces are assumed to be Tychonoff (= completely regular and $T_1$). The closure of a subset $A$ of a topological space $X$ is denoted by $\overline{A}$, $\overline{A}^X$ or $\cl_X(A)$. A topological space $X$ is defined to be {\em selectively sequentially pseudocompact} if for any sequence $\{U_n\}_{n\in\w}$ of open sets of $X$ there exists a sequence $(x_n)_{n\in\w}\in\prod_{n\in\w}U_n$ containing a convergent subsequence.
A function $f:X\to Y$ between topological spaces $X$ and $Y$ is called {\em sequentially continuous} if for any convergent sequence $\{x_n\}_{n\in\w}\subseteq X$, the sequence $\{f(x_n)\}_{n\in\w}$ converges in $Y$ and $\lim_{n}f(x_n)=f(\lim_{n}x_n)$. We denote by $C(X)$ the vector space of all continuous $\ff$-valued functions on $X$.
A subset $A$ of a topological space $X$ is called
\begin{enumerate}
\item[$\bullet$] {\em relatively compact} if its closure ${\bar A}$ is compact;
\item[$\bullet$] ({\em relatively}) {\em sequentially compact} if each sequence in $A$ has a subsequence converging to a point of $A$ (resp., of $X$);
\item[$\bullet$] {\em functionally bounded in $X$} if every $f\in C(X)$ is bounded on $A$.
\end{enumerate}
The space $C(X)$ endowed with the pointwise topology is denoted by $C_p(X)$.

Let $E$ be a locally convex space. We assume that $E$ is over the field $\ff$ of real or complex numbers.  We denote by $\Nn_0(E)$ (resp., $\Nn_{0}^c(E)$) the family of all (resp., closed absolutely convex) neighborhoods of zero of $E$. The family of all bounded subsets of $E$ is denoted by $\Bo(E)$. The topological dual space of $E$  is denoted by $E'$. The value of $\chi\in E'$ on $x\in E$ is denoted by $\langle\chi,x\rangle$ or $\chi(x)$. A sequence $\{x_n\}_{n\in\w}$ in $E$ is said to be {\em Cauchy} if for every $U\in\Nn_0(E)$ there is $N\in\w$ such that $x_n-x_m\in U$ for all $n,m\geq N$. It is easy to see that a sequence $\{x_n\}_{n\in\w}$ in  $E$ is Cauchy if and only if $x_{n_k}-x_{n_{k+1}}\to 0$ for every (strictly) increasing sequence $(n_k)$ in $\w$. We denote by $E_w$ and $E_\beta$ the space $E$ endowed with the weak topology $\sigma(E,E')$ and with the strong topology $\beta(E,E')$, respectively. The topological dual space $E'$ of $E$ endowed with weak$^\ast$ topology $\sigma(E',E)$ or with the strong topology $\beta(E',E)$ is denoted by $E'_{w^\ast}$ or $E'_\beta$, respectively. The closure of a subset $A$ in the weak topology is denoted by $\overline{A}^{\,w}$ or $\overline{A}^{\,\sigma(E,E')}$, and $\overline{B}^{\,w^\ast}$ (or $\overline{B}^{\,\sigma(E',E)}$) denotes the closure of $B\subseteq E'$ in the weak$^\ast$ topology. The {\em polar} of a subset $A$ of $E$ is denoted by
\[
A^\circ :=\{ \chi\in E': \|\chi\|_A \leq 1\}, \quad\mbox{ where }\quad\|\chi\|_A=\sup\big\{|\chi(x)|: x\in A\cup\{0\}\big\}.
\]
A subset $B$ of $E'$ is {\em equicontinuous} if $B\subseteq U^\circ$ for some $U\in \Nn_0(E)$. The family of all continuous linear maps (= operators) from an lcs $H$ to an lcs $L$ is denoted by $\LL(H,L)$.

Let $p\in[1,\infty]$. Then $p^\ast$ is defined to be the unique element of $ [1,\infty]$ which satisfies $\tfrac{1}{p}+\tfrac{1}{p^\ast}=1$. For $p\in[1,\infty)$, the space $\ell_{p^\ast}$  is the dual space of $\ell_p$. We denote by $\{e_n\}_{n\in\w}$ the canonical basis of $\ell_p$, if $1\leq p<\infty$, or the canonical basis of $c_0$, if $p=\infty$. The canonical basis of $\ell_{p^\ast}$ is denoted by $\{e_n^\ast\}_{n\in\w}$.  Denote by  $\ell_p^0$ and $c_0^0$ the linear span of $\{e_n\}_{n\in\w}$  in  $\ell_p$ or $c_0$ endowed with the induced norm topology, respectively.

A subset $A$ of a locally convex space $E$ is called
\begin{enumerate}
\item[$\bullet$] {\em precompact} if for every $U\in\Nn_0(E)$ there is a finite set $F\subseteq E$ such that $A\subseteq F+U$;
\item[$\bullet$] {\em sequentially precompact} if every sequence in $A$ has a Cauchy subsequence;
\item[$\bullet$] {\em weakly $($sequentially$)$ compact} if $A$ is (sequentially) compact in $E_w$;
\item[$\bullet$] {\em relatively weakly compact} if its weak closure $\overline{A}^{\,\sigma(E,E')}$ is compact in $E_w$;
\item[$\bullet$] {\em relatively weakly sequentially compact} if each sequence in $A$ has a subsequence weakly converging to a point of $E$;
\item[$\bullet$] {\em weakly sequentially precompact} if each sequence in $A$ has a weakly Cauchy subsequence.
\end{enumerate}
Note that each sequentially precompact subset of $E$ is precompact, but the converse is not true in general, see Lemma 2.2 of \cite{Gab-Pel}. 

Let $p\in[1,\infty]$. A sequence  $\{x_n\}_{n\in\w}$ in a locally convex space $E$ is called
\begin{enumerate}
\item[$\bullet$]  {\em weakly $p$-summable} if  for every $\chi\in E'$, it follows
\[
\mbox{$(\langle\chi, x_n\rangle)_{n\in\w} \in\ell_p$ if $p<\infty$, and  $(\langle\chi, x_n\rangle)_{n\in\w} \in c_0$ if $p=\infty$;}
\]
\item[$\bullet$] {\em weakly $p$-convergent to $x\in E$} if  $\{x_n-x\}_{n\in\w}$ is weakly $p$-summable;
\item[$\bullet$] {\em weakly $p$-Cauchy} if for each pair of strictly increasing sequences $(k_n),(j_n)\subseteq \w$, the sequence  $(x_{k_n}-x_{j_n})_{n\in\w}$ is weakly $p$-summable.
\end{enumerate}
The family of all weakly $p$-summable sequences of $E$ is denoted by $\ell_p^w(E)$ or $c_0^w(E)$ if $p=\infty$.

A sequence  $\{\chi_n\}_{n\in\w}$ in $E'$ is called  {\em weak$^\ast$ $p$-summable} (resp., {\em weak$^\ast$ $p$-convergent to $\chi\in E'$} or  {\em weak$^\ast$ $p$-Cauchy})  if it is weakly $p$-summable (resp., weakly $p$-convergent to $\chi\in E'$ or weakly $p$-Cauchy) in $E'_{w^\ast}$.

Generalizing the classical notions of limited subsets, $p$-limited subsets, $p$-$(V^\ast)$ subsets and coarse $p$-limited subsets of a Banach space $X$ and $p$-$(V)$ subsets of the Banach dual $X^\ast$ introduced in \cite{BourDies}, \cite{KarnSinha}, \cite{CCDL}, \cite{GalMir} and \cite{LCCD}, respectively, we defined in \cite{Gab-Pel,Gab-limited} the following notions.
Let $1\leq p\leq q\leq\infty$, and let $E$ be a locally convex space $E$. Then:
\begin{enumerate}
\item[$\bullet$] a non-empty subset   $A$ of $E$ is a  {\em $(p,q)$-$(V^\ast)$ set} (resp., a {\em $(p,q)$-limited set}) if
\[
\Big(\sup_{a\in A} |\langle \chi_n, a\rangle|\Big)\in \ell_q \; \mbox{ if $q<\infty$, } \; \mbox{ or }\;\; \Big(\sup_{a\in A} |\langle \chi_n, a\rangle|\Big)\in c_0 \; \mbox{ if $q=\infty$},
\]
for every  weakly (resp., weak$^\ast$) $p$-summable sequence $\{\chi_n\}_{n\in\w}$ in  $E'_\beta$. $(p,\infty)$-$(V^\ast)$ sets and $(1,\infty)$-$(V^\ast)$ sets will be called simply {\em $p$-$(V^\ast)$ sets} and {\em $(V^\ast)$ sets}, respectively. Analogously, $(p,p)$-limited sets and $(\infty,\infty)$-limited sets will be called {\em $p$-limited sets} and {\em limited sets}, respectively.
\item[$\bullet$] a non-empty subset   $A$ of $E$ is a {\em coarse $p$-limited set} if for every $T\in\LL(E,\ell_p) $ $($or $T\in\LL(E,c_0)$ if $p=\infty$), the set $T(A)$ is relatively compact.
\item[$\bullet$] a non-empty subset $B$ of $E'$ is a  {\em $(p,q)$-$(V)$ set} if
\[
\Big(\sup_{\chi\in B} |\langle \chi, x_n\rangle|\Big)\in \ell_q \; \mbox{ if $q<\infty$, } \; \mbox{ or }\;\; \Big(\sup_{\chi\in B} |\langle \chi, x_n\rangle|\Big)\in c_0 \; \mbox{ if $q=\infty$},
\]
for every weakly $p$-summable sequence $\{x_n\}_{n\in\w}$ in  $E$. $(p,\infty)$-$(V)$ sets and $(1,\infty)$-$(V)$ sets  will be called simply {\em $p$-$(V)$ sets} and {\em $(V)$ sets}, respectively.
\end{enumerate}

Recall that a locally convex space  $E$
\begin{enumerate}
\item[$\bullet$] is {\em sequentially complete} if each Cauchy sequence in $E$ converges;
\item[$\bullet$] ({\em quasi}){\em barrelled} if every $\sigma(E',E)$-bounded (resp., $\beta(E',E)$-bounded) subset of $E'$ is equicontinuous;
\item[$\bullet$] {\em $c_0$-}({\em quasi}){\em barrelled} if every $\sigma(E',E)$-null (resp., $\beta(E',E)$-null) sequence is equicontinuous.
\end{enumerate}

The following weak barrelledness conditions introduced and studied in \cite{Gab-Pel} will play a considerable role in the article.
Let $p\in[1,\infty]$. A locally convex space $E$ is called  {\em $p$-barrelled } (resp.,  {\em $p$-quasibarrelled}) if every weakly $p$-summable sequence in $E'_{w^\ast}$ (resp., in  $E'_\beta$) is equicontinuous. It is clear that $E$ is $\infty$-barrelled if and only if it is $c_0$-barrelled.

We shall consider also the following linear map introduced in \cite{Gab-Pel}
\[
S_p:\LL(E,\ell_p)\to \ell_p^w(E'_{w^\ast}) \quad \big(\mbox{or $S_\infty: \LL(E,c_0)\to c_0^w(E'_{w^\ast})$ if $p=\infty$}\big)
\]
defined by $S_p(T):=\big(T^\ast(e^\ast_n)\big)_{n\in\w}$.

The following $p$-versions of weakly compact-type properties are defined in \cite{Gab-Pel} generalizing the corresponding notions in the class of Banach spaces introduced in \cite{CS} and \cite{Ghenciu-pGP}. Let $p\in[1,\infty]$. A subset   $A$ of a locally convex space $E$ is called
\begin{enumerate}
\item[$\bullet$] ({\em relatively}) {\em weakly sequentially $p$-compact} if every sequence in $A$ has a weakly $p$-convergent  subsequence with limit in $A$ (resp., in $E$);
\item[$\bullet$] {\em weakly  sequentially $p$-precompact} if every sequence from $A$ has a  weakly $p$-Cauchy subsequence.
\end{enumerate}
It is clear that each relatively weakly sequentially $p$-compact subset of $E$ is weakly  sequentially $p$-precompact.

Let $E$ and $L$ be locally convex spaces. An operator $T\in \LL(E,L)$ is called
{\em weakly sequentially compact} (resp., {\em  weakly sequentially $p$-compact, weakly sequentially $p$-precompact {\em or} coarse $p$-limited}) if there is $U\in \Nn_0(E)$ such that $T(U)$ is a  relatively weakly sequentially compact  (resp., relatively weakly sequentially $p$-compact, weakly  sequentially $p$-precompact or coarse $p$-limited)  subset of $L$. Generalizing the notion of $p$-convergent operators between Banach spaces and following \cite{Gab-Pel}, an operator $T\in \LL(E,L)$ is called {\em $p$-convergent} if $T$ sends weakly $p$-summable sequences of $E$ to null sequences of $L$.


\section{Main results} \label{sec:limited-p-conv}


The following  assertion  gives the first relationships between different $p$-convergent types of operators. Recall that a locally convex space $E$ is called {\em Grothendieck} or has the {\em Grothendieck property} if the identity map $\Id_{E'}:  E'_{w^\ast} \to \big(E'_\beta\big)_w$ is sequentially continuous.
\begin{proposition} \label{p:weak*-p-convergent-weak-con}
Let $p\in[1,\infty]$,  $E$ and $L$ be locally convex spaces, and let $T:E\to L$ be a linear map. Then:
\begin{enumerate}
\item[{\rm(i)}] if $T$ is finite-dimensional and continuous, then $T$ is $p$-convergent, coarse $p$-limited and  weak$^\ast$ $p$-convergent;
\item[{\rm(ii)}]  if $T$ is weak$^\ast$ $p$-convergent, then it is weakly $p$-convergent; the converse it true if $L$ has the Grothendieck property;
\item[{\rm(iii)}] if $L$ is $\infty$-quasibarrelled and $T$ is $p$-convergent, then $T$ is weakly $p$-convergent;
\item[{\rm(iv)}] if $L$ is $c_0$-barrelled and  $T$ is $p$-convergent, then $T$ is weak$^\ast$ $p$-convergent.
\end{enumerate}
\end{proposition}

\begin{proof}
(i) follows from the corresponding definitions and (iv) of Proposition 4.2 of \cite{Gab-limited} 
(which states that every finite subset of $E$ is coarse $p$-limited).

(ii) follows from the fact that every weakly null sequence in $L'_\beta$ is also weak$^\ast$ null and the definition of the Grothendieck property.

(iii), (iv):  Let $\{\eta_n\}_{n\in\w}\subseteq L'_\beta$ be a weakly (resp., weak$^\ast$) null-sequence, and let $\{x_n\}_{n\in\w}\subseteq E$ be a weakly $p$-summable sequence. As $L$ is $\infty$-quasibarrelled (resp., $c_0$-barrelled), the sequence $\{\eta_n\}_{n\in\w}$ is equicontinuous. Now, fix an arbitrary $\e>0$. Choose $U\in \Nn_0(L)$ such that
\begin{equation} \label{equ:weak-p-conver-2}
|\langle\eta_n,y\rangle|<\e \quad \mbox{ for every $n\in\w$ and each $y\in U$}.
\end{equation}
Since $T$ is $p$-convergent, $T(x_n)\to 0$ in $L$, and hence there is $N_\e\in\w$ such that
\begin{equation} \label{equ:weak-p-conver-3}
T(x_n)\in U \quad \mbox{ for every $n\geq N_\e$}.
\end{equation}
Then (\ref{equ:weak-p-conver-2}) and  (\ref{equ:weak-p-conver-3}) imply
\[
|\langle\eta_n,T(x_n)\rangle| <\e \quad \mbox{ for every $n\geq N_\e$}.
\]
Therefore $\langle\eta_n,T(x_n)\rangle\to 0$ as $n\to\infty$. Thus $T$ is weakly (resp., weak$^\ast$) $p$-convergent. \qed
\end{proof}

Below we consider the case when the classes of all weakly $p$-convergent, weak$^\ast$ $p$-convergent and $p$-convergent operators coincide. Observe that the conditions of the proposition are satisfied if $L$ is a separable reflexive Fr\'{e}chet space or a reflexive Banach space. In particular, this proposition generalizes Corollary 2.4 and Proposition 2.5 of \cite{FZ-p}.
\begin{proposition} \label{p:weak*-p-convergent=weak-con}
Let $p\in[1,\infty]$, $E$ be a locally convex space, and let $L$ be an $\infty$-quasibarrelled Grothendieck space such that $U^\circ$ is weak$^\ast$ selectively sequentially pseudocompact for every $U\in\Nn_0(L)$. Then for an operator $T:E\to L$, the following assertions are equivalent:
\begin{enumerate}
\item[{\rm(i)}] $T$ is weakly $p$-convergent;
\item[{\rm(ii)}] $T$ is weak$^\ast$ $p$-convergent;
\item[{\rm(iii)}] $T$ is  $p$-convergent.
\end{enumerate}
\end{proposition}

\begin{proof}
The equivalence (i)$\Leftrightarrow$(ii) follow from (ii) of Proposition \ref{p:weak*-p-convergent-weak-con}, and the implication (iii)$\Ra$(i) follows from (iii) of Proposition \ref{p:weak*-p-convergent-weak-con}.

(ii)$\Ra$(iii) Suppose for a contradiction that there is a weakly $p$-summable sequence $\{x_n\}_{n\in\w}$ in $E$ such that $T(x_n)\not\to 0$. Without loss of generality we can assume that there is $V\in\Nn_0^c(L)$ such that $T(x_n)\not\in V$ for every $n\in\w$. By the Hahn--Banach separation theorem, for every $n\in\w$ there is $\eta_n\in V^\circ$ such that $|\langle\eta_n, T(x_n)\rangle|>1$. For every $n\in\w$, set
\[
U_n:= \{ \chi\in V^\circ: |\langle\chi, T(x_n)\rangle|>1\}.
\]
Then $U_n$ is a weak$^\ast$ open neighborhood of $\eta_n$ in $V^\circ$. Since $V^\circ$ is  selectively sequentially pseudocompact in the weak$^\ast$ topology, for every $n\in\w$ there exists $\chi_n\in U_n$ such that the sequence $\{\chi_n\}_{n\in\w}$ contains a subsequence $\{\chi_{n_k}\}_{k\in\w}$ which weak$^\ast$ converges to some functional $\chi\in V^\circ$. Taking into account that the subsequence $\{x_{n_k}\}_{k\in\w}$ is also weakly $p$-summable and the operator $T$ is  weak$^\ast$ $p$-convergent the inclusion $\chi_{n_k}\in U_{n_k}$ implies
\[
1< \langle\chi_{n_k}, T(x_{n_k})\rangle = \langle \chi_{n_k}-\chi,T(x_{n_k})\rangle + \langle \chi, T(x_{n_k})\rangle\to 0 \; \mbox{ as } k\to\infty,
\]
a contradiction.\qed
\end{proof}

In (iv) of Proposition \ref{p:weak*-p-convergent-weak-con} the condition of being a $c_0$-barrelled space is not necessary in general even for operators, but  this condition cannot be completely omitted even for metrizable spaces, and also the condition in (iii) of being $\infty$-quasibarrelled is essential as the following example shows.

\begin{example} \label{exa:weak*-p-convergent-weak-con}
Let $p\in[1,\infty]$.
\begin{enumerate}
\item[{\rm(i)}] There are metrizable non-$c_0$-barrelled spaces $E$ and $L$ such that each operator $T:E\to L$ is finite-dimensional and hence it is $p$-convergent and  weak$^\ast$ $p$-convergent.
\item[{\rm(ii)}] There is a metrizable non-$c_0$-barrelled space $E$ such that the identity map $\Id_E$ is $p$-convergent and weakly  $p$-convergent, but it does not weak$^\ast$ $p$-convergent.
\item[{\rm(iii)}] There are a non-$\infty$-quasibarrelled space $E$ and an $\infty$-convergent operator $T:E\to E$ such that $T$ is not weakly $\infty$-convergent.
\end{enumerate}
\end{example}

\begin{proof}
(i) Let $E=(c_0)_p$ be the Banach space $c_0$ endowed with the topology induced from $\IF^\w$, and let $L:=c_0^0$. Then $E$ and $L$ are metrizable and non-barrelled. Since metrizable locally convex spaces are barrelled if and only if they are $c_0$-barrelled (see Proposition 12.2.3 of \cite{Jar}), it remains to note that, by (ii) of Example 5.4 of \cite{Gab-Pel}, 
each $T\in\LL(E,L)$  is finite-dimensional.

(ii) Let $E=L=C_p(\mathbf{s})$, where $\mathbf{s}=\{x_n\}_{n\in\w} \cup\{x_\infty\}$ is a convergent sequence. Then $E$ is a metrizable space and hence quasibarrelled. By Proposition 12.2.3 of \cite{Jar} 
and the Buchwalter–-Schmets theorem, $E$ is not $c_0$-barrelled.  Since $E$ carries its weak topology, $\Id_E$ is trivially $p$-convergent for every $p\in[1,\infty]$. Therefore, by (iii) of Proposition \ref{p:weak*-p-convergent-weak-con}, $\Id_E$ is weakly $p$-convergent. 
To show that $\Id_E$ is not  weak$^\ast$ $p$-convergent, for every $n\in\w$, let $\eta_n:=\delta_{x_n} -\delta_{x_\infty}$ and $f_n:= 1_{\{x_n\}}$, where $\delta_x$ is the Dirac measure at the point $x$ and $1_A$ is the characteristic function of a subset $A\subseteq \mathbf{s}$. It is well known that $C_p(\mathbf{s})'=L(\mathbf{s})$ algebraically, where $L(\mathbf{s})$ is the free locally convex space over $\mathbf{s}$. Now it is clear that the sequence $\{\eta_n\}_{n\in\w}$ is weak$^\ast$ null in $L'$ and $(f_n)\in \ell_p^w(E)$ (or $\in c_0^w(E)$ if $p=\infty$). However, since $\langle\eta_n,\Id_E(f_n)\rangle=f_n(x_n)-f_n(x_\infty)=1\not\to 0$ we obtain that $\Id_E$ is not weak$^\ast$ $p$-convergent.

(iii) Let $\mathbf{s}=\{x_n\}_{n\in\w}\cup\{ x_\infty\}$ be a convergent sequence, and let $E=L(\mathbf{s})$ be the free locally convex space over $\mathbf{s}$. By Example 5.5 of \cite{Gab-Pel},  
the space $E$ is not $1$-quasibarrelled and hence it is not $\infty$-quasibarrelled. Let $T=\Id_{E}:E\to E$ be the identity operator. Since, by Theorem 1.2 of \cite{Gab-Respected}, $E$ has the Schur property the operator $T$ is trivially $p$-convergent. We show that $T$ is not weakly $\infty$-convergent. To this end, as in the proof of (ii), we consider two sequences $\{\eta_n:=\delta_{x_n} -\delta_{x_\infty}\}_{n\in\w}\subseteq E$ and $\{f_n\}_{n\in\w}\subseteq E'=C(\mathbf{s})$. It is clear that  $\{\eta_n\}_{n\in\w}$ is weakly null in $E$. Since, by Proposition 3.4 of \cite{Gabr-free-lcs}, the space $E'_\beta$ is the Banach space $C(\mathbf{s})$, it is easy to see that the sequence $\{f_n=1_{\{x_n\}}\}_{n\in\w}$ is weakly null (= weakly $\infty$-summable). Taking into account that \[
\langle f_n, T(\eta_n)\rangle=\langle f_n, \delta_{x_n}-\delta_{x_\infty}\rangle= f_n(x_n)-f_n(x_\infty)=1 \not\to 0,
\]
it follows that $T$ is not weakly $\infty$-convergent.\qed
\end{proof}

Below we generalize the classes of limited, limited completely continuous and  limited $p$-convergent operators between Banach spaces.
\begin{definition} \label{def:limited-p-convergent} {\em
Let $p,q,q'\in[1,\infty]$, $q'\leq q$, and let $E$ and $L$ be locally convex spaces. A linear map $T:E\to L$ is called
\begin{enumerate}
\item[$\bullet$] {\em weakly $(p,q)$-convergent}  if
$
\lim_{n\to\infty} \langle\eta_n, T(x_n)\rangle=0
$
for every weakly $q$-summable sequence $\{\eta_n\}_{n\in\w}$ in $L'_\beta$ and each weakly $p$-summable sequence $\{x_n\}_{n\in\w}$ in $E$;
\item[$\bullet$] {\em weak$^\ast$ $(p,q)$-convergent} if
$
\lim_{n\to\infty} \langle\eta_n, T(x_n)\rangle=0
$
for every weak$^\ast$  $q$-summable sequence $\{\eta_n\}_{n\in\w}$ in $L'$ and each weakly $p$-summable sequence $\{x_n\}_{n\in\w}$ in $E$;
\item[$\bullet$] {\em $(p,q)$-limited} if $T(U)$ is a $(p,q)$-limited subset of $L$ for some $U\in\Nn_0(E)$; if $p=q$ or $p=q=\infty$ we shall say that $T$ is {\em $p$-limited} or {\em limited}, respectively;
\item[$\bullet$] {\em $(p,q)$-$(V^\ast)$} if $T(U)$ is a $(p,q)$-$(V^\ast)$ subset of $L$ for some $U\in\Nn_0(E)$; if $q=\infty$ or $p=1$ and $q=\infty$ we shall say that $T$ is {\em $p$-$(V^\ast)$} or {\em $(V^\ast)$}, respectively;
\item[$\bullet$] {\em $(q',q)$-limited $p$-convergent} if $T(x_n)\to 0$ for every  weakly $p$-summable sequence $\{x_n\}_{n\in\w}$ in $E$ which is a $(q',q)$-limited subset of $E$; if $q'=q$ or $q'=q=\infty$ we shall say that $T$ is {\em $q$-limited $p$-convergent} or {\em limited $p$-convergent}, respectively;
\item[$\bullet$] {\em $(q',q)$-$(V^\ast)$ $p$-convergent} if $T(x_n)\to 0$ for every  weakly $p$-summable sequence $\{x_n\}_{n\in\w}$ in $E$ which is a $(q',q)$-$(V^\ast)$ subset of $E$; if $q=\infty$ or $q'=q=\infty$ we shall say that $T$ is {\em $q'$-$(V^\ast)$ $p$-convergent} or {\em $(V^\ast)$ $p$-convergent}, respectively.    \qed
\end{enumerate} }
\end{definition}
It is clear that weakly $p$-convergent operators are exactly weakly $(p,\infty)$-convergent operators, and weak$^\ast$ $p$-convergent operators are exactly weak$^\ast$ $(p,\infty)$-convergent operators.

\begin{lemma} \label{l:fd-pq-limited}
Let $p,q,q'\in[1,\infty]$, $q'\leq q$, and let $E$ and $L$ be locally convex spaces. If $T\in\LL(E,L)$ is finite-dimensional, then $T$ is $(p,q)$-limited, $(p,q)$-$(V^\ast)$, $(q',q)$-limited $p$-convergent and $(q',q)$-$(V^\ast)$ $p$-convergent.
\end{lemma}

\begin{proof}
Since $T$ is finite-dimensional, there are a finite subset $F$ of $L$ and $U\in \Nn_0(E)$ such that $T(U)\subseteq\cacx(F)$. Therefore, by Lemma 3.1 of \cite{Gab-limited}, 
$T(U)$ is a $(p,q)$-limited set and hence a  $(p,q)$-$(V^\ast)$ set. Whence $T$ is $(p,q)$-limited and  $(p,q)$-$(V^\ast)$.
Since, by (i) of Proposition \ref{p:weak*-p-convergent-weak-con}, $T$ is $p$-convergent it is trivially $(q',q)$-limited $p$-convergent and $(q',q)$-$(V^\ast)$ $p$-convergent.\qed
\end{proof}

The next proposition stands ideal properties of the classes of operators introduced in Definition \ref{def:limited-p-convergent}.

\begin{proposition} \label{p:weak-pq-convergent-ideal}
Let $p,q,q'\in[1,\infty]$,  $q'\leq q$,  $\lambda,\mu\in\IF$, the spaces $E_0$, $E$, $L_0$ and $L$ be locally convex, and let $Q\in\LL(E_0,E)$, $T,S\in\LL(E,L)$ and $R\in\LL(L,L_0)$. Then:
\begin{enumerate}
\item[{\rm(i)}] if $T$ and $S$ are weakly $(p,q)$-convergent, then so are $R\circ T\circ Q$ and $\lambda T+\mu S$;
\item[{\rm(ii)}] if $T$  and $S$ are  weak$^\ast$ $(p,q)$-convergent, then so are $R\circ T\circ Q$ and $\lambda T+\mu S$;
\item[{\rm(iii)}] if $T$  and $S$ are  $(p,q)$-limited operators, then so are $R\circ T\circ Q$ and $\lambda T+\mu S$;
\item[{\rm(iv)}] if $T$  and $S$ are  $(p,q)$-$(V^\ast)$ operators, then so are $R\circ T\circ Q$ and $\lambda T+\mu S$;
\item[{\rm(v)}] if $T$  and $S$ are  $(q',q)$-limited $p$-convergent, then so are $R\circ T\circ Q$ and $\lambda T+\mu S$;
\item[{\rm(vi)}] if $T$  and $S$ are  $(q',q)$-$(V^\ast)$ $p$-convergent, then so are $R\circ T\circ Q$ and $\lambda T+\mu S$;
\item[{\rm(vii)}] if $T$  and $S$ are  coarse $p$-limited, then so are $R\circ T\circ Q$ and $\lambda T+\mu S$.
\end{enumerate}
\end{proposition}

\begin{proof}
Since the case $\lambda T+\mu S$ is trivial, we consider the case $R\circ T\circ Q$.

(i) and (ii): Let $\{\eta_n\}_{n\in\w}$ be a weakly (resp., weak$^\ast$) $q$-summable sequence in $(L_0)'_\beta$,  and let $\{x_n\}_{n\in\w}$ be a  weakly $p$-summable sequence in $E_0$. Then, by (iii) of Lemma 4.6 of \cite{Gab-Pel}, 
the sequence $\{Q(x_n)\}_{n\in\w}$ is weakly $p$-summable in $E$. Since $R^\ast$ is weak$^\ast$ and  strongly continuous by  Theorems 8.10.5 and 8.11.3 of \cite{NaB}, respectively, Lemma 4.6 of of \cite{Gab-Pel} 
implies that the sequence $\{R^\ast(\eta_n)\}_{n\in\w}$ is weakly (resp., weak$^\ast$) $q$-summable in $L'_\beta$.
Now the definition of weakly (resp., weak$^\ast$  $(p,q)$-convergent operators  implies 
\[
\lim_{n\to\infty} \langle \eta_n,R\circ T\circ Q(x_n) \rangle=\lim_{n\to\infty} \langle R^\ast(\eta_n), T\big( Q(x_n)\big) \rangle=0.
\]
Thus $R\circ T\circ Q$ is weakly (resp., weak$^\ast$) $(p,q)$-convergent, as desired.
\smallskip

(iii) immediately follows from (iv) of Lemma 3.1 of \cite{Gab-limited}. 

(iv) immediately follows from  (iv) of Lemma 7.2 of \cite{Gab-Pel}. 
\smallskip

(v) and (vi): Let $\{x_n\}_{n\in\w}$ be a  weakly $p$-summable sequence in $E_0$ which is a $(q',q)$-limited (resp., $(q',q)$-$(V^\ast)$) subset of $E$. Then, by  Lemma 4.6 of of \cite{Gab-Pel} 
and (iv) of Lemma 3.1 of \cite{Gab-limited} 
 (resp., (iv) of Lemma 7.2 of \cite{Gab-Pel}), 
the sequence $\{Q(x_n)\}_{n\in\w}$ is a  weakly $p$-summable sequence in $E$ which is a $(q',q)$-limited (resp., $(q',q)$-$(V^\ast)$) subset of $E$. Since $T$ is $(q',q)$-limited (resp., $(q',q)$-$(V^\ast)$) $p$-convergent, it follows that $T(Q(x_n))\to 0$ in the space $L$. The continuity of $R$ implies that $R\circ T\circ Q(x_n)\to 0$ in $L_0$. Thus $R\circ T\circ Q$ is $(q',q)$-limited (resp., $(q',q)$-$(V^\ast)$) $p$-convergent.
\smallskip

(vii) immediately follows from (iii) of Lemma 4.1 of \cite{Gab-limited} 
(which states that continuous images of coarse $p$-limited sets are coarse $p$-limited). \qed
\end{proof}

\begin{corollary} \label{c:weak-p-convergent}
Let $p,q,q'\in[1,\infty]$,  $q'\leq q$, and let $E$ and $L$ be locally convex spaces. If the identity operator $\Id_E:E\to E$ is weakly $(p,q)$-convergent $($resp., weak$^\ast$ $(p,q)$-convergent, $(q',q)$-limited, $(q',q)$-$(V^\ast)$,  $(q',q)$-limited $p$-convergent, $(q',q)$-$(V^\ast)$ $p$-convergent or coarse $p$-limited$)$, then so is every $T\in\LL(E,L)$.
\end{corollary}

\begin{proof}
The assertion follows from the equality $T=T\circ \Id_E \circ \Id_E$ and Proposition \ref{p:weak-pq-convergent-ideal}. \qed
\end{proof}

Corollary \ref{c:weak-p-convergent} motivates the study of locally convex spaces $E$ for which the identity operator $\Id_E:E\to E$ has one of the properties from the corollary. If the space $L$ contains an isomorphic copy of $\ell_\infty$ we can partially reverse Corollary \ref{c:weak-p-convergent} as follows. 
\begin{theorem} \label{t:weak-p-conv-limited}
Let $p,q,q'\in[1,\infty]$,  $q'\leq q$, and let  $E$ and $L$ be locally convex spaces. If $L$ contains an isomorphic copy of $\ell_\infty$, then the following assertions are equivalent:
\begin{enumerate}
\item[{\rm(i)}]   each weakly $p$-convergent operator from $E$ into $L$ is $(q',q)$-limited $p$-convergent;
\item[{\rm(ii)}] each weak$^\ast$ $p$-convergent operator from $E$ into $L$ is $(q',q)$-limited $p$-convergent;
\item[{\rm(iii)}] the identity operator $\Id_E: E\to E$ is $(q',q)$-limited $p$-convergent.
\end{enumerate}
\end{theorem}

\begin{proof}
(i)$\Ra$(ii) follows from (ii) of Proposition \ref{p:weak*-p-convergent-weak-con}.

(ii)$\Ra$(iii) Suppose for a contradiction  that $\Id_E$ is not  $(q',q)$-limited $p$-convergent. Then there is a  $(q',q)$-limited weakly $p$-summable sequence $S=\{x_n\}_{n\in\w}$ in $E$ which does not converge to zero in $E$. Without loss of generality we can assume that $S\cap U=\emptyset$ for some closed absolutely convex neighborhood $U$ of zero in $E$. Since $S$, being also a weakly null-sequence, is a bounded subset of $E$, there is $a>1$ such that $S\subseteq aU$. For every $n\in\w$, by the Hahn--Banach separation theorem, choose $\chi_n\in U^\circ$ such that $|\langle\chi_n,x_n\rangle|>1$. Therefore
\begin{equation} \label{equ:weak-p-conver-1}
1< |\langle\chi_n,x_n\rangle|\leq a \;\; \mbox{ for every }\; n\in\w.
\end{equation}
Since $\{\chi_n\}_{n\in\w}\subseteq U^\circ$, the sequence $\{\chi_n\}_{n\in\w}$ is equicontinuous. Therefore, by Lemma 14.13 of \cite{Gab-Pel}, 
the linear map
\[
Q: E\to \ell_\infty, \quad Q(x):= \big( \langle\chi_n,x\rangle\big)_{n\in\w},
\]
is continuous.

By assumption, there is an embedding $R:\ell_\infty\to L$. Consider the operator $T:=R\circ Q: E\to L$. We claim that $T$ is weak$^\ast$ $p$-convergent. Indeed, let $(y_n)_{n\in\w}\in \ell_p^w(E)$  (or $(y_n)_{n\in\w}\in c_0^w(E)$  if $p=\infty$) and let $\{\eta_n\}_{n\in\w}$ be a weak$^\ast$ null-sequence in $(\ell_\infty)'$. Then $\{ R^\ast(\eta_n)\}_{n\in\w}$ is also a weak$^\ast$ null-sequence in $(\ell_\infty)'$. Therefore, by the Grothendieck property of $\ell_\infty$, $\{ R^\ast(\eta_n)\}_{n\in\w}$ is weakly null in the Banach dual space $(\ell_\infty)'_\beta$. Since $\{y_n\}_{n\in\w}$ is weakly null, it follows that $\{Q(y_n)\}_{n\in\w}\subseteq \ell_\infty$ is also weakly null. Therefore, by the Dunford--Pettis property of $\ell_\infty$, we obtain
\[
\lim_{n\to\infty} \langle\eta_n, T(y_n)\rangle=\lim_{n\to\infty} \langle R^\ast(\eta_n), Q(y_n)\rangle=0.
\]
Thus $T$ is a weak$^\ast$ $p$-convergent operator.

To get a desired contradiction it remains to prove that $T$ is not $(q',q)$-limited $p$-convergent by showing that $T(x_k)\not\to 0$. To this end, choose $W\in\Nn_0(L)$ such that $W\cap R(\ell_\infty) \subseteq R(B_{\ell_\infty})$. The inequalities (\ref{equ:weak-p-conver-1}) and the bijectivity of $R$ imply
\[
T(x_k)=R\circ Q(x_k)= R\big( \langle\chi_n,x_k\rangle\big)\not\in R(B_{\ell_\infty}) \;\; \mbox{ for every } \; k\in\w.
\]
Since the range of $T$ is contained in $R(\ell_\infty)$ and $R$ is an embedding the choice of $W$ implies that $T(x_k)\not\in W$ for all $k\in\w$. Thus $T$ is not $(q',q)$-limited $p$-convergent.

(iii)$\Ra$(i) follows from Corollary \ref{c:weak-p-convergent}. \qed
\end{proof}

 Corollary \ref{c:weak-p-convergent}  and Theorem \ref{t:weak-p-conv-limited} immediately imply
\begin{corollary} \label{c:Id-p-limited}
Let $p,q,q'\in[1,\infty]$,  $q'\leq q$, and let  $E$ be a locally convex space. Then the identity operator $\Id_E: E\to E$ is $(q',q)$-limited $p$-convergent if and only if so is any operator $T:E\to\ell_\infty$.
\end{corollary}

Below we obtain an analogous  characterization of locally convex spaces $E$ for which the identity map $\Id_E$ is $(q',q)$-$(V^\ast)$ $p$-convergent. We omit its proof because it can be obtained from the proof of  Theorem \ref{t:weak-p-conv-limited} just replacing ``$(q',q)$-limited'' by ``$(q',q)$-$(V^\ast)$''.

\begin{theorem} \label{t:weak-p-conv-V*}
Let $p,q,q'\in[1,\infty]$,  $q'\leq q$,  and let  $E$ and $L$ be locally convex spaces. If $L$ contains an isomorphic copy of $\ell_\infty$, then the following assertions are equivalent:
\begin{enumerate}
\item[{\rm(i)}]   each weakly $p$-convergent operator from $E$ into $L$ is $(q',q)$-$(V^\ast)$ $p$-convergent;
\item[{\rm(ii)}] each weak$^\ast$ $p$-convergent operator from $E$ into $L$ is $(q',q)$-$(V^\ast)$ $p$-convergent;
\item[{\rm(iii)}] the identity operator $\Id_E: E\to E$ is $(q',q)$-$(V^\ast)$ $p$-convergent.
\end{enumerate}
\end{theorem}

 Corollary \ref{c:weak-p-convergent}  and Theorem \ref{t:weak-p-conv-V*} immediately imply
\begin{corollary} \label{c:Id-p-V*}
Let $p,q,q'\in[1,\infty]$,  $q'\leq q$, and let  $E$ be a locally convex space. Then the identity operator $\Id_E: E\to E$ is $(q',q)$-$(V^\ast)$ $p$-convergent if and only if so is any operator $T:E\to\ell_\infty$.
\end{corollary}


Let $p\in[1,\infty]$. We shall say that a locally convex space $E$ is ({\em weakly}) {\em sequentially  locally  $p$-complete} if the closed absolutely convex hull of a weakly $p$-summable sequence is weakly sequentially $p$-compact (resp., weakly sequentially $p$-precompact). It is clear that if $p=\infty$ and $E$ is weakly angelic (for example, $E$ is a strict $(LF)$-space), then $E$ is sequentially  locally  $\infty$-complete if and only if it is locally complete.

Now we characterize weak$^\ast$ $p$-convergent operators.

\begin{theorem} \label{t:weak*-p-convergent}
Let $p\in[1,\infty]$, $E$ and $L$ be locally convex spaces,  and let  $T:E\to L$ be an operator. 
Consider the following assertions:
\begin{enumerate}
\item[{\rm(i)}] $T$ is weak$^\ast$ $p$-convergent;
\item[{\rm(ii)}] $T$ transforms weakly sequentially $p$-precompact subsets of $E$ to limited subsets of $L$;
\item[{\rm(iii)}] $T$ transforms {\rm(}relatively{\rm)} weakly sequentially $p$-compact subsets of $E$ to limited subsets of $L$;
\item[{\rm(iv)}] $T$ transforms weakly $p$-summable sequence of $E$ to limited subsets of $L$;
\item[{\rm(v)}] $S\circ T$ is $p$-convergent for each $S\in \LL(L,Z)$ and any locally convex {\rm(}or the same, Banach{\rm)} space $Z$ such that $U^\circ$ is weak$^\ast$ selectively sequentially pseudocompact for every  $U\in\Nn_0(Z)$; 
\item[{\rm(vi)}] $S\circ T$ is $p$-convergent for each $S\in \LL(L,c_0)$;
\item[{\rm(vii)}] for any normed space $X$ and each weakly sequentially $p$-precompact operator $R:X\to E$, the operator $T\circ R$ is limited;
\item[{\rm(viii)}] for any normed space $X$ and each weakly sequentially $p$-precompact operator $R:X\to E$, the adjoint $R^\ast \circ T^\ast: L'_{w^\ast}\to X'_\beta$ is $\infty$-convergent;
\item[{\rm(ix)}] if $R\in \LL(\ell_1^0,E)$ is weakly sequentially $p$-precompact, then $T\circ R$ is limited;
\item[{\rm(x)}]  for every normed space $Z$ and each weakly sequentially $p$-compact operator $S$ from $Z$ to $E$, the composition $T\circ S$ is a limited linear map;
\item[{\rm(xi)}]  for any operator $S\in \LL(\ell_{p^\ast}, E)$, the linear map $T\circ S$ is limited.
\end{enumerate}
Then:
\begin{enumerate}
\item[{\rm(A)}] {\rm(i)$\Leftrightarrow$(ii)$\Leftrightarrow$(iii)$\Leftrightarrow$(iv)};
\item[{\rm(B)}] {\rm(i)$\Ra$(v)$\Ra$(vi)}, and if $L$ is $c_0$-barrelled, then {\rm(vi)$\Ra$(i)};
\item[{\rm(C)}] {\rm(ii)$\Ra$(vii)$\Leftrightarrow$(viii)$\Ra$(ix)}, and if $E$ is weakly sequentially locally $p$-complete, then {\rm(ix)$\Ra$(i)};
\item[{\rm(D)}] if $1<p<\infty$, then {\rm(iii)$\Ra$(x)$\Ra$(xi)}, and if $E$ is sequentially complete, then {\rm(xi)$\Ra$(i)}.
\end{enumerate}
\end{theorem}

\begin{proof}
(i)$\Ra$(ii) Let $A$ be a weakly sequentially $p$-precompact subset of $E$, and suppose for a contradiction that $T(A)$ is not limited. Therefore there are a weak$^\ast$ null sequence $\{\eta_n\}_{n\in\w}$ in $L'$, a sequence $\{x_n\}_{n\in\w}$ in  $A$ and $\e>0$ such that $|\langle\eta_n,T(x_n)\rangle |>\e$ for every $n\in\w$. Since $A$ is weakly sequentially $p$-precompact, without loss of generality we assume that $\{x_n\}_{n\in\w}$ is weakly $p$-Cauchy.

For $n_0=0$, since $\{\eta_n\}_{n\in\w}$ is weak$^\ast$ null we can choose $n_1>n_0$ such that $|\langle\eta_{n_1},T(x_{n_0})\rangle |<\tfrac{\e}{2}$. Proceeding by induction on $k$, we can choose $n_{k+1}>n_k$ such that $|\langle\eta_{n_{k+1}},T(x_{n_k})\rangle |<\tfrac{\e}{2}$. Since the sequence $\{x_{n_{k+1}} -x_{n_k}\}_{k\in\w}$ is weakly $p$-summable and $T$ is weak$^\ast$ $p$-convergent, we obtain
\[
\big\langle \eta_{n_{k+1}}, T\big(x_{n_{k+1}} -x_{n_k}\big)\big\rangle \to 0.
\]
On the other hand,
\[
\big|\big\langle \eta_{n_{k+1}}, T\big(x_{n_{k+1}} -x_{n_k}\big)\big\rangle\big| \geq \big|\big\langle \eta_{n_{k+1}}, T\big(x_{n_{k+1}}\big)\big\rangle\big| - \big|\big\langle \eta_{n_{k+1}}, T\big(x_{n_k}\big)\big\rangle\big|> \tfrac{\e}{2},
\]
a contradiction.
\smallskip

(ii)$\Ra$(iii) and (iii)$\Ra$(iv) are trivial.
\smallskip

(iv)$\Ra$(i) follows from the definition of limited sets.
\smallskip

(i)$\Ra$(v) Assume that $T$ is weak$^\ast$ $p$-convergent, and suppose for a contradiction that $S\circ T$ is not $p$-convergent for some $S\in \LL(L,Z)$ and some locally convex  (resp., Banach) space $Z$ such that $U^\circ$ is weak$^\ast$ selectively sequentially pseudocompact for every $U\in\Nn_0(Z)$. Then there are a weakly $p$-summable sequence $\{x_n\}_{n\in\w}$ in $E$ and a closed absolutely convex neighborhood $V\subseteq Z$ of zero such that $ST(x_n)\not\in V$ for every $n\in\w$. By the Hahn--Banach separation theorem, for every $n\in\w$ there is $\eta_n\in V^\circ$ such that $|\langle\eta_n, ST(x_n)\rangle|>1$. For every $n\in\w$, set
\[
U_n:= \{ \chi\in V^\circ: |\langle\chi, ST(x_n)\rangle|>1\}.
\]
Then $U_n$ is a weak$^\ast$ open neighborhood of $\eta_n$ in $V^\circ$. Since $V^\circ$ is  selectively sequentially pseudocompact in the weak$^\ast$ topology, for every $n\in\w$ there exists $\chi_n\in U_n$ such that the sequence $\{\chi_n\}_{n\in\w}$ contains a subsequence $\{\chi_{n_k}\}_{k\in\w}$ which weak$^\ast$ converges to some functional $\chi\in V^\circ$. By Theorem 8.10.5  of \cite{NaB}, the adjoint operator $S^\ast$ is weak$^\ast$ continuous and hence $S^\ast(\chi_{n_k}) \to S^\ast(\chi)$ in the weak$^\ast$ topology.  Since $T$ is weak-weak sequentially continuous, we have $ST(x_{n_k})\to 0$ in the weak topology. Taking into account that $T$ is weak$^\ast$ $p$-convergent, we obtain
\[
|\langle\chi_{n_k}, ST(x_{n_k})\rangle|\leq |\langle S^\ast(\chi_{n_k}-\chi),T(x_{n_k})\rangle| + |\langle \chi, ST(x_{n_k})\rangle|\to 0 \; \mbox{ as } k\to\infty.
\]
Since, by the choice of $\chi_{n_k}$, $|\langle\chi_{n_k}, ST(x_{n_k})\rangle|>1$ we obtain a desired contradiction.
\smallskip

(v)$\Ra$(vi) follows from the well known fact that $B_{c_0}^\circ$ is even a weak$^\ast$ metrizable compact space.
\smallskip

(vi)$\Ra$(i) Assume that $L$ is $c_0$-barrelled. To show that $T$ is weak$^\ast$ $p$-convergent, fix a weakly $p$-summable sequence $\{x_n\}_{n\in\w}$ in $E$ and a weak$^\ast$ null sequence $\{\chi_n\}_{n\in\w}$ in $L'$. Since $L$ is $c_0$-barrelled,  (ii) of Proposition 4.17 of \cite{Gab-Pel} 
implies that there is $S\in\LL(L,c_0)$ such that $S(y)=\big(\langle\chi_n,y\rangle\big)_{n\in\w}$ for every $y\in L$. By assumption $ST$ is $p$-convergent. Therefore $ST(x_n)\to 0$ in $c_0$. Taking into account that $\{e^\ast_n\}_{n\in\w}$ is bounded in the Banach space $(c_0)'=\ell_1$, we obtain
\[
\begin{aligned}
\big|\langle\chi_k, T(x_k)\rangle\big| & =\big|\langle e^\ast_k, (\langle\chi_n, T(x_k)\rangle)_{n\in\w}\rangle\big|
=\big|\langle e^\ast_k, ST(x_k)\rangle\big|\\
& \leq \|e^\ast_k\|_{\ell_1} \cdot \|ST(x_k)\|_{c_0} = \|ST(x_k)\|_{c_0}\to 0.
\end{aligned}
\]
This shows that $T$ is weak$^\ast$ $p$-convergent.
\smallskip

(ii)$\Ra$(vii) Assume that  $T$ transforms weakly sequentially $p$-precompact subsets of $E$ to limited subsets of $L$, and let  $R:X\to E$ be a weakly sequentially $p$-precompact operator. Then 
the set $R(B_X)$ is weakly sequentially $p$-precompact, and hence $TR(B_X)$ is a limited subset of $L$. Thus the operator $T\circ R$ is limited.
\smallskip

(vii)$\Leftrightarrow$(viii) follows from the equivalence (i)$\Leftrightarrow$(ii) of Theorem 5.5 
of \cite{Gab-limited} applied to $T\circ R$ and $p=\infty$.
\smallskip

(vii)$\Ra$(ix) is obvious.
\smallskip

(ix)$\Ra$(i) Assume additionally that $E$ is weakly sequentially locally $p$-complete. Let $S=\{x_n\}_{n\in\w}$ be a weakly $p$-summable sequence in $E$. Then, by Proposition 14.9 of \cite{Gab-Pel}, 
the linear map  $R:\ell_1^0 \to E$ defined by
\[
R(a_0 e_0+\cdots+a_ne_n):=a_0 x_0+\cdots+ a_n x_n \quad (n\in\w, \; a_0,\dots,a_n\in\IF)
\]
is continuous. It is clear that $R\big(B_{\ell_1^0}\big) \subseteq \cacx(S)$. Since $E$ is weakly sequentially locally $p$-complete it follows that $\cacx(S)$ is weakly sequentially $p$-precompact. Therefore $R$ is a weakly sequentially $p$-precompact operator and hence, by (ix), $TR$ is limited. Whence for every weak$^\ast$ null sequence $\{\eta_n\}_{n\in\w}$ in $L'$ we obtain
\[
|\langle\eta_n,T(x_n)\rangle|=|\langle\eta_n,TR(e_n)\rangle|\leq \sup_{x \in B_{\ell_1^0}} |\langle\eta_n,TR(x)\rangle|\to 0 \mbox{ as $n\to\infty$}.
\]
Thus $T$ is weak$^\ast$ $p$-convergent.
\smallskip


(iii)$\Ra$(x) Let $Z$ be a normed space, and let $S:Z\to E$ be a weakly sequentially $p$-compact operator. Then $S(B_Z)$ is a relatively weakly sequentially $p$-compact subset of $E$. By (iii), the set $TS(B_Z)$ is limited. Thus the linear map $T\circ S$ is limited.
\smallskip

(x)$\Ra$(xi) 
By Proposition 1.4 of \cite{CS} (or by (ii) of Corollary 13.11 of \cite{Gab-Pel}), 
the identity operator $\Id_{\ell_{p^\ast}}$ of $\ell_{p^\ast}$ is weakly sequentially $p$-compact. Hence 
each operator $S=S\circ \Id_{\ell_{p^\ast}}\in \LL(\ell_{p^\ast},E)$ is weakly sequentially $p$-compact. Thus, by (x), $T\circ S$ is limited for every $S\in \LL(\ell_{p^\ast},E)$.
\smallskip

(xi)$\Ra$(i) Assume that $E$ is sequentially complete. Let $\{\chi_n\}_{n\in\w}$ be a weak$^\ast$ null sequence in $L'$, and let $\{x_n\}_{n\in\w}$ be a weakly $p$-summable sequence in $E$. By Proposition 4.14 of \cite{Gab-Pel}, 
there is $S\in \LL(\ell_{p^\ast},E)$ such that $S(e_n^\ast)=x_n$ for every $n\in\w$ (where $\{e^\ast_n\}_{n\in\w}$ is the canonical unit basis of $\ell_{p^\ast}$). By (xi), $T\circ S$ is limited. Therefore
\[
\big|\langle\chi_{n}, T(x_n)\rangle\big|= \big|\langle\chi_{n}, TS(e_n^\ast)\rangle\big|\leq \sup_{x\in B_{\ell_{p^\ast}}} |\langle\chi_{n}, TS(x)\rangle|\to 0 \quad \mbox{ as }\; n\to\infty.
\]
Thus the linear map $T$ is weak$^\ast$ $p$-convergent.\qed
\end{proof}

\begin{remark} {\em
The condition on $L$ to be $c_0$-barrelled in the implication  (vi)$\Ra$(i) in (B) of Theorem \ref{t:weak*-p-convergent} is essential. Indeed, let $p\in[1,\infty]$, $E=L=C_p(\mathbf{s})$ and let $T=\Id:E\to L$ be the identity map. Then, by (ii) of Example \ref{exa:weak*-p-convergent-weak-con}, $L$ is a metrizable non-$c_0$-barrelled space and $T$ is not weak$^\ast$ $p$-convergent. On the other hand, it is easy to see that each $S\in \LL(L,c_0)$ is finite-dimensional (see Lemma 17.18 of \cite{Gab-Pel}), and therefore $S\circ T$ is also finite-dimensional and hence $p$-convergent (see (i) of Proposition \ref{p:weak*-p-convergent-weak-con}). Thus the implication  (vi)$\Ra$(i) in (B) of Theorem \ref{t:weak*-p-convergent} does not hold.\qed}
\end{remark}

Below we characterize weakly $p$-convergent operators  (for the definition of the map $S_\infty$ see Section \ref{sec:pre}).

\begin{theorem} \label{t:char-weakly-p-convergent}
Let $p\in[1,\infty]$, $E$ and $L$ be locally convex spaces,  and let  $T:E\to L$ be an operator. 
Consider the following assertions:
\begin{enumerate}
\item[{\rm(i)}] $T$ is weakly $p$-convergent;
\item[{\rm(ii)}] $T$ transforms weakly sequentially $p$-precompact subsets of $E$ to  $\infty$-$(V^\ast)$ subsets of $L$;
\item[{\rm(iii)}] $T$ transforms weakly sequentially $p$-compact subsets of $E$ to  $\infty$-$(V^\ast)$ subsets of $L$;
\item[{\rm(iv)}] $T$ transforms weakly $p$-summable sequences of $E$ to  $\infty$-$(V^\ast)$ subsets of $L$;
\item[{\rm(v)}] $R\circ T$ is $p$-convergent for  any Banach space $Z$ and each $R\in \LL(L,Z)$ with weakly sequentially precompact adjoint $R^\ast:Z'_\beta \to L'_\beta$; 
\item[{\rm(vi)}] $R\circ T$ is $p$-convergent for each $R\in \LL(L,c_0)$ with weakly sequentially precompact adjoint $R^\ast:\ell_1 \to L'_\beta$;
\item[{\rm(vii)}] $R\circ T$ is $p$-convergent for each operator $R\in \LL(L,c_0)$ such that $S_\infty(R)=(\chi_n)$ is weakly null in $L'_\beta$;
\item[{\rm(viii)}] for any normed space $X$ and each weakly sequentially $p$-precompact operator $R:X\to E$, the operator $T\circ R$ is $\infty$-$(V^\ast)$;
\item[{\rm(ix)}]  for every normed space $X$ and each weakly sequentially $p$-precompact operator $S$ from $X$ to $E$, the map $S^\ast\circ T^\ast$ is $\infty$-convergent;
\item[{\rm(x)}] if $R\in \LL(\ell_1^0,E)$ is weakly sequentially $p$-precompact, then $T\circ R$ is $\infty$-$(V^\ast)$;
\item[{\rm(xi)}]  for every normed space $Z$ and each weakly sequentially $p$-compact operator $S$ from $Z$ to $E$, the composition $T\circ S$ is a $\infty$-$(V^\ast)$ linear map;
\item[{\rm(xii)}]  for any operator $S\in \LL(\ell_{p^\ast}, E)$, the linear map $T\circ S$ is $\infty$-$(V^\ast)$.
\end{enumerate}
Then:
\begin{enumerate}
\item[{\rm(A)}] {\rm(i)$\Leftrightarrow$(ii)$\Leftrightarrow$(iii)$\Leftrightarrow$(iv)};
\item[{\rm(B)}] {\rm(i)$\Ra$(v)$\Ra$(vi)}, and if $L$ is $c_0$-barrelled and $L'_\beta$ is weakly sequentially  locally  $\infty$-complete, then {\rm(vi)$\Ra$(i)};
\item[{\rm(C)}] {\rm(i)$\Ra$(vii)}, and if $L$ is $c_0$-barrelled, then {\rm(vii)$\Ra$(i)};
\item[{\rm(D)}] {\rm(ii)$\Ra$(viii)$\Leftrightarrow$(ix)$\Ra$(x)}, and if $E$ is weakly sequentially locally $p$-complete, then {\rm(x)$\Ra$(i)};
\item[{\rm(E)}] if $1<p<\infty$, then {\rm(i)$\Ra$(viii)$\Ra$(xi)$\Ra$(xii)}, and if $E$ is sequentially complete, then {\rm(xii)$\Ra$(i)}.
\end{enumerate}
\end{theorem}

\begin{proof}
(i)$\Ra$(ii) Let $A$ be a weakly sequentially $p$-precompact subset of $E$, and suppose for a contradiction that $T(A)$ is not an $\infty$-$(V^\ast)$ set. Therefore there is a weakly null sequence $\{\eta_n\}_{n\in\w}$ in $L'_\beta$, a sequence $\{x_n\}_{n\in\w}$ in  $A$ and $\e>0$ such that $|\langle\eta_n,T(x_n)\rangle |>\e$ for every $n\in\w$. Since $A$ is weakly sequentially $p$-precompact, without loss of generality we assume that $\{x_n\}_{n\in\w}$ is weakly $p$-Cauchy.

For $n_0=0$, since $\{\eta_n\}_{n\in\w}$ is also weak$^\ast$ null we can choose $n_1>n_0$ such that $|\langle\eta_{n_1},T(x_{n_0})\rangle |<\tfrac{\e}{2}$. Proceeding by induction on $k$, we can choose $n_{k+1}>n_k$ such that $|\langle\eta_{n_{k+1}},T(x_{n_k})\rangle |<\tfrac{\e}{2}$. Since the sequence $\{x_{n_{k+1}} -x_{n_k}\}_{k\in\w}$ is weakly $p$-summable and $T$ is weakly $p$-convergent, we obtain
\[
\big\langle \eta_{n_{k+1}}, T\big(x_{n_{k+1}} -x_{n_k}\big)\big\rangle \to 0.
\]
On the other hand, for every $k\in\w$,  we have
\[
\big|\big\langle \eta_{n_{k+1}}, T\big(x_{n_{k+1}} -x_{n_k}\big)\big\rangle\big| \geq \big|\big\langle \eta_{n_{k+1}}, T\big(x_{n_{k+1}}\big)\big\rangle\big| - \big|\big\langle \eta_{n_{k+1}}, T\big(x_{n_k}\big)\big\rangle\big|> \tfrac{\e}{2},
\]
a contradiction.

(ii)$\Ra$(iii) and (iii)$\Ra$(iv) are trivial.

(iv)$\Ra$(i) follows from the definition of $\infty$-$(V^\ast)$.

(i)$\Ra$(v) Assume that $T$ is weakly $p$-convergent, and suppose for a contradiction that $R\circ T$ is not $p$-convergent for some  Banach space $Z$ and $R\in \LL(L,Z)$ with weakly sequentially precompact adjoint $R^\ast:Z'_\beta \to L'_\beta$.  Then there are a weakly $p$-summable sequence $\{x_n\}_{n\in\w}$ in $E$ and $\delta>0$  such that $\|RT(x_n)\|>\delta$ for every $n\in\w$. By the Hahn--Banach separation theorem, for every $n\in\w$ there is $\eta_n\in (\delta B_Z)^\circ$ such that $|\langle R^\ast(\eta_n), T(x_n)\rangle|=|\langle\eta_n, RT(x_n)\rangle|>1$. Since $R^\ast$ is weakly sequentially precompact, passing to subsequences if needed we can assume that the sequence $\{R^\ast(\eta_n)\}_{n\in\w}\subseteq L'_\beta$ is weakly Cauchy.

For $n_0=0$, choose $n_1> n_0$ such that $|\langle R^\ast(\eta_{n_0}), T(x_{n_1})\rangle|<\tfrac{1}{2}$ (this is possible since $\{x_n\}_{n\in\w}$ and hence also $\{T(x_n)\}_{n\in\w}$ are weakly null). By induction on $k\in\w$, for every $k>0$ choose $n_{k+1}> n_k$ such that $|\langle R^\ast(\eta_{n_k}), T(x_{n_{k+1}})\rangle|<\tfrac{1}{2}$. As $\{R^\ast(\eta_n)\}_{n\in\w}\subseteq L'_\beta$ is weakly Cauchy, the sequence $\{R^\ast(\eta_{n_k})-R^\ast(\eta_{n_{k+1}})\}_{k\in\w}$ is weakly null in $L'_\beta$. Taking into account that $\{x_{n_k}\}_{n\in\w}$ is also weakly $p$-summable and $T$ is weakly $p$-convergent we obtain
\[
0\leftarrow \big| \big\langle R^\ast(\eta_{n_k}-\eta_{n_{k+1}}), T(x_{n_{k+1}})\big\rangle\big| \geq
\big| \big\langle R^\ast(\eta_{n_{k+1}}), T(x_{n_{k+1}})\big\rangle\big|-\big| \big\langle R^\ast(\eta_{n_{k}}), T(x_{n_{k+1}})\big\rangle\big|>\tfrac{1}{2},
\]
a contradiction.
\smallskip

(v)$\Ra$(vi) is trivial.
\smallskip

(vi)$\Ra$(i) Assume that $L$ is $c_0$-barrelled and $L'_\beta$ is weakly sequentially  locally  $\infty$-complete. To show that $T$ is weakly $p$-convergent, fix a weakly $p$-summable sequence $\{x_n\}_{n\in\w}$ in $E$ and a weakly null sequence $\{\chi_n\}_{n\in\w}$ in $L'_\beta$. Since $L$ is $c_0$-barrelled,  (ii) of Proposition 4.17 of \cite{Gab-Pel} 
implies that there is $R\in\LL(L,c_0)$ such that $R(y)=\big(\langle\chi_n,y\rangle\big)_{n\in\w}$ for every $y\in L$. Since $\{\chi_n\}_{n\in\w}$ is weakly null in $L'_\beta$ and $L'_\beta$ is weakly sequentially  locally  $\infty$-complete,  it follows that $\cacx\big(\{\chi_n\}_{n\in\w}\big)$ is weakly sequentially precompact. Taking into account that $R^\ast(e_n^\ast)=\chi_n$ for every $n\in\w$ (where as usual $\{e_n^\ast\}_{n\in\w}$ is the canonical unit basis of $\ell_1$), we obtain $R^\ast(B_{\ell_1})\subseteq \cacx\big(\{\chi_n\}_{n\in\w}\big)$. Therefore $R^\ast$ is  weakly sequentially precompact, and hence, by (vi),  $RT$ is $p$-convergent. Hence $RT(x_n)\to 0$ in $c_0$. Therefore
\[
\big|\langle\chi_k, T(x_k)\rangle\big|=\big|\langle e^\ast_k, (\langle\chi_n, T(x_k)\rangle)_{n\in\w}\rangle\big|
=\big|\langle e^\ast_k, RT(x_k)\rangle\big|\leq \|e^\ast_k\|_{\ell_1} \cdot \|RT(x_k)\|_{c_0} = \|RT(x_k)\|_{c_0}\to 0.
\]
Thus $T$ is weakly $p$-convergent.
\smallskip

(i)$\Ra$(vii) Assume that $T$ is weakly $p$-convergent, and suppose for a contradiction that $R\circ T$ is not $p$-convergent for some operator $R\in \LL(L,c_0)$ such that the sequence $S_\infty(R)=(\chi_n)$ is weakly null in $L'_\beta$.  Then there are a weakly $p$-summable sequence $\{x_n\}_{n\in\w}$ in $E$ and $\e>0$  such that $\|RT(x_{n})\|_{c_0}\geq \e$ for every $n\in\w$. Recall that $R(y)=\big(\langle\chi_n,y\rangle\big)_n \in c_0$ for every $y\in L$. Then for every $n\in\w$, we have (where as usual $\{e^\ast_i\}_{i\in\w}$ is the canonical unit basis of $\ell_1=(c_0)'$)
\begin{equation} \label{equ:p-convergent-10}
\e \leq\|RT(x_{n})\|_{c_0} = \sup_{i\in\w} |\langle e^\ast_i, RT(x_{n})\rangle|=\sup_{i\in\w} |\langle R^\ast(e^\ast_i),T(x_n)\rangle|=\sup_{i\in\w} |\langle \chi_i, T(x_n)\rangle|.
\end{equation}
For $n_0=0$, choose $i_0\in\w$ such that $|\langle \chi_{i_0}, T(x_{n_0})\rangle|\geq \tfrac{\e}{2}$. Since $T$ is weak-weak sequentially continuous and because the sequence $\{x_n\}_{n\in\w}$ is also weakly null, it follows that $T(x_n)\to 0$ in the weak topology of $L$. Therefore we can choose $n_1>n_0$ such that
\begin{equation} \label{equ:p-convergent-11}
|\langle \chi_{i}, T(x_{n_1})\rangle|< \tfrac{\e}{2}\;\; \mbox{ for every $i\leq i_0$}.
\end{equation}
By (\ref{equ:p-convergent-10}), there is $i_1\in\w$ such that $|\langle \chi_{i_1}, T(x_{n_1})\rangle|\geq \tfrac{\e}{2}$. Taking into account (\ref{equ:p-convergent-11}) we obtain that $i_1>i_0$. Since  $T(x_n)\to 0$ in the weak topology of $L$, there exists $n_2>n_1$ such that
\begin{equation} \label{equ:p-convergent-12}
|\langle \chi_{i}, T(x_{n_2})\rangle|< \tfrac{\e}{2}\;\; \mbox{ for every $i\leq i_1$}.
\end{equation}
By (\ref{equ:p-convergent-10}), there is $i_2\in\w$ such that $|\langle \chi_{i_2}, T(x_{n_2})\rangle|\geq \tfrac{\e}{2}$. By (\ref{equ:p-convergent-12}), we obtain that $i_2>i_1$. Continuing this process we find two sequences $\{\chi_{i_k}\}_{k\in\w}$ and  $\{T(x_{n_k})\}_{n\in\w}$ such that $\{i_k\}_{k\in\w}$ and $\{n_k\}_{k\in\w}$ are strictly increasing and $|\langle \chi_{i_k}, T(x_{n_k})\rangle|\geq \tfrac{\e}{2}$ for every $k\in\w$. Clearly, the sequence $\{\chi_{i_k}\}_{k\in\w}$ is weakly null in $L'_\beta$ and the sequence $\{x_{n_k}\}_{n\in\w}$ is weakly $p$-summable in $E$.  Then the weak $p$-convergence of $T$ and the choice of these two sequences imply
\[
\tfrac{\e}{2}\leq \lim_{k\to\infty} |\langle \chi_{i_k}, T(x_{n_k})\rangle|= 0
\]
which is impossible.
\smallskip

(vii)$\Ra$(i) To show that $T$ is weakly $p$-convergent, fix a weakly $p$-summable sequence $\{x_n\}_{n\in\w}$ in $E$ and a weakly null sequence $\{\chi_n\}_{n\in\w}$ in $L'_\beta$. Since $L$ is $c_0$-barrelled,  (ii) of Proposition 4.17 of \cite{Gab-Pel} 
implies that there is $R\in\LL(L,c_0)$ such that $R(y)=\big(\langle\chi_n,y\rangle\big)_{n\in\w}$ for every $y\in L$. Since $\{\chi_n\}_{n\in\w}$ is weakly null in $L'_\beta$ and $S_\infty(R)=(\chi_n)$,  (vii) implies that  $RT$ is $p$-convergent. Hence $RT(x_n)\to 0$ in $c_0$. If $\{e^\ast_n\}_{n\in\w}$ is the canonical unit basis of $(c_0)'=\ell_1$, we obtain
\[
\big|\langle\chi_k, T(x_k)\rangle\big|=\big|\langle e^\ast_k, (\langle\chi_n, T(x_k)\rangle)_{n\in\w}\rangle\big|
=\big|\langle e^\ast_k, RT(x_k)\rangle\big|\leq \|e^\ast_k\|_{\ell_1} \cdot \|RT(x_k)\|_{c_0} = \|RT(x_k)\|_{c_0}\to 0.
\]
Thus $T$ is weakly $p$-convergent.
\smallskip

(ii)$\Ra$(viii) Assume that  $T$ transforms weakly sequentially $p$-precompact subsets of $E$ to $\infty$-$(V^\ast)$ subsets of $L$, and let  $R:X\to E$ be a weakly sequentially $p$-precompact operator. Then 
the set $R(B_X)$ is weakly sequentially $p$-precompact, and hence $TR(B_X)$ is an $\infty$-$(V^\ast)$ subset of $L$. Thus the operator $T\circ R$ is $\infty$-$(V^\ast)$.

(viii)$\Leftrightarrow$(ix) follows from the equivalence (i)$\Leftrightarrow$(ii) in Theorem 14.1 of \cite{Gab-Pel} 
applied to $T\circ R$ and $p=\infty$.

(viii)$\Ra$(x) is obvious.

(x)$\Ra$(ii) Assume that $E$ is weakly sequentially locally $p$-complete. Let $S=\{x_n\}_{n\in\w}$ be a weakly $p$-summable sequence in $E$. Then, by Proposition 14.9 of \cite{Gab-Pel}, 
the linear map  $R:\ell_1^0 \to E$ defined by
\[
R(a_0 e_0+\cdots+a_ne_n):=a_0 x_0+\cdots+ a_n x_n \quad (n\in\w, \; a_0,\dots,a_n\in\IF)
\]
is continuous. It is clear that $R\big(B_{\ell_1^0}\big) \subseteq \cacx(S)$. Since $E$ is weakly sequentially locally $p$-complete it follows that $\cacx(S)$ is weakly sequentially $p$-precompact. Therefore $R$ is a weakly sequentially $p$-precompact operator and hence, by (x), $TR$ is $\infty$-$(V^\ast)$. Whence for every weakly null sequence $\{\eta_n\}_{n\in\w}$ in $L'_\beta$ we obtain
\[
|\langle\eta_n,T(x_n)\rangle|=|\langle\eta_n,TR(e_n)\rangle|\leq \sup_{x\in B_{\ell_1^0}} |\langle\eta_n,TR(x)\rangle|\to 0 \mbox{ as $n\to\infty$}.
\]
Thus $T$ is weakly $p$-convergent.
\smallskip


(viii)$\Ra$(xi) is trivial.
\smallskip

(xi)$\Ra$(xii) Let $1<p<\infty$. 
By Proposition 1.4 of \cite{CS} (or by (ii) of Corollary 13.11 of \cite{Gab-Pel}), 
the identity operator $\Id_{\ell_{p^\ast}}$ of $\ell_{p^\ast}$ is weakly sequentially $p$-compact. Hence 
each operator $S=S\circ \Id_{\ell_{p^\ast}}\in \LL(\ell_{p^\ast},E)$ is weakly sequentially $p$-compact. Thus, by (xi), $T\circ S$ is an $\infty$-$(V^\ast)$ operator for every $S\in \LL(\ell_{p^\ast},E)$.
\smallskip

(xii)$\Ra$(i) Let $1<p<\infty$ and assume that $E$ is sequentially complete. Let $\{\chi_n\}_{n\in\w}$ be a weakly null sequence in $L'_\beta$, and let $\{x_n\}_{n\in\w}$ be a weakly $p$-summable sequence in $E$. By Proposition 4.14 of \cite{Gab-Pel}, 
there is $S\in \LL(\ell_{p^\ast},E)$ such that $S(e_n^\ast)=x_n$ for every $n\in\w$ (where $\{e^\ast_n\}_{n\in\w}$ is the canonical unit basis of $\ell_{p^\ast}$). By (xii), $T\circ S$ is a $\infty$-$(V^\ast)$ map. Therefore
\[
\big|\langle\chi_{n}, T(x_n)\rangle\big|= \big|\langle\chi_{n}, TS(e_n^\ast)\rangle\big|\leq \sup_{x\in B_{\ell_{p^\ast}}} |\langle\chi_{n}, TS(x)\rangle|\to 0 \quad \mbox{ as }\; n\to\infty.
\]
Thus the linear map $T$ is weakly $p$-convergent.\qed
\end{proof}

\begin{remark} {\em
The condition on $E$ to be sequentially complete in (D) of Theorem \ref{t:weak*-p-convergent} and in (E) of Theorem \ref{t:char-weakly-p-convergent} is essential. Indeed, let $1<p<\infty$, $E=\ell_{p^\ast}^0$, $L=\ell_q$ with $p^\ast\leq q<\infty$,  and let $T=\Id:E\to L$ be the identity inclusion.
Then every $S\in \LL(\ell_{p^\ast},E)$ is finite-dimensional (indeed, since $E=\bigcup_{n\in\w} \IF^n$ it follows that $\ell_{p^\ast}=\bigcup_{n\in\w} S^{-1}(\IF^n)$ and hence, by the Baire property of $\ell_{p^\ast}$, $S^{-1}(\IF^m)$ is an open linear subspace of $\ell_{p^\ast}$ for some $m\in\w$ that is possible only if $S^{-1}(\IF^m)=\ell_{p^\ast}$; so $S$ is finite-dimensional). Therefore, by Lemma \ref{l:fd-pq-limited}, $T\circ S$ is limited for each $S\in \LL(\ell_{p^\ast},E)$. However, $T$ is not weakly $p$-convergent and hence, by (ii) of Proposition \ref{p:weak*-p-convergent-weak-con}, $T$ is not weak$^\ast$ $p$-convergent. Indeed, for every $n\in\w$, let $x_n=e^\ast_n\in E$ and $\eta_n=e^\ast_n\in L'$. Then the sequence $\{x_n\}_{n\in\w}$ is weakly $p$-summable in $E$ (for every $\chi=(a_n)\in \ell_p=E'$, we have $\sum_{n\in\w} |\langle \chi,x_n\rangle|^p=\sum_{n\in\w} |a_n|^p<\infty$) and the sequence  $\{\eta_n\}_{n\in\w}$ is even weakly $q$-summable in $\ell_{q^\ast} =L'_\beta$ (for every $x=(b_n)\in \ell_q=(L'_\beta)'$, we have $\sum_{n\in\w} |\langle x,\eta_n\rangle|^q=\sum_{n\in\w} |b_n|^q<\infty$). Since for every $n\in\w$, $\langle\eta_n,T(x_n)\rangle=1$ it follows that $T$ is not weakly $p$-convergent.
\qed}
\end{remark}

\bibliographystyle{amsplain}

\begin{thebibliography}{10}


\bibitem{AB}
C.D. Aliprantis, O. Burkinshaw, {\em Positive Operators}, Springer, Dordrecht, 2006.


\bibitem{BourDies}
J. Bourgain, J. Diestel, {\em Limited operators and strict cosingularity}, Math. Nachr. \textbf{119} (1984), 55--58.



\bibitem{CS}
J.M.F. Castillo, F. Sanchez, Dunford--Pettis-like properties of continuous vector function spaces, Revista Mat. Complut. \textbf{6}:1 (1993), 43--59.

\bibitem{CCDL}
D. Chen, J.A. Ch\'{a}vez-Dom\'{\i}nguez, L. Li, $p$-Converging operators and Dunford--Pettis property of order $p$, J. Math. Anal. Appl. \textbf{461} (2018), 1053--1066.


\bibitem{EKHBM}
A. El Kaddouri, J. H'michane, K. Bouras, M. Moussa, On the class of weak$^\ast$ Dunford--Pettis operators, Rend. Circ. Mat. Palermo \textbf{62}:2 (2013), 261--265.


\bibitem{FZ-con}
J.H. Fourie, E.D. Zeekoei, On $p$-convergent operators on Banach spaces and Banach lattices, FABWI-N-WSK-2015-461, North-West University, Potchefstroom, 2015, 27 pp.




\bibitem{FZ-p}
J.H. Fourie, E.D. Zeekoei, On weak-star $p$-convergent operators, Quest. Math. \textbf{40} (2017), 563--579.

\bibitem{FZ-pL}
J.H. Fourie, E.D. Zeekoei, On $p$-convergent operators on Banach lattices, Acta Math. Sinica  \textbf{34}:5 (2018), 873--890.



\bibitem{Gabr-free-lcs}
S. Gabriyelyan, Locally convex properties of free locally convex spaces,  J. Math. Anal. Appl. 480 (2019) 123453.


\bibitem{Gab-Respected}
S. Gabriyelyan, {\em Maximally almost periodic groups and respecting properties}, Descriptive Topology and Functional Analysis II, J.C. Ferrando (ed.), Springer Proceedings in Mathematics \& Statistics,  \textbf{286} (2019), 103--136.


\bibitem{Gab-Pel}
S. Gabriyelyan, Pe{\l}czy\'{n}ski's type sets and Pe{\l}czy\'{n}ski's geometrical properties of locally convex spaces, submitted  ({\tt arxiv.org/abs/2402.08860}).


\bibitem{Gab-limited}
S. Gabriyelyan, Limited type subsets of locally convex spaces, submitted  ({\tt arxiv.org/abs/2403.02016}).



\bibitem{GalMir}
P. Galindo, V.C.C. Miranda, A class of sets in a Banach space coarser than limited sets, Bull. Braz. Math. Soc., New series \textbf{53} (2022), 941--955.

\bibitem{Ghenciu-pGP}
I. Ghenciu, The $p$-Gelfand--Phillips property in spaces of operators and Dunford--Pettis like sets, Acta Math. Hungar. \textbf{155}:2 (2018), 439--457.

\bibitem{Jar}
H.~Jarchow, \emph{Locally Convex Spaces}, B.G. Teubner, Stuttgart, 1981.

\bibitem{KarnSinha}
A.K. Karn, D.P. Sinha, An operator summability in Banach spaces, Glasgow Math. J. \textbf{56} (2014), 427--437.

\bibitem{LCCD}
L. Li, D. Chen, J.A. Ch\'{a}vez-Dom\'{\i}nguez, Pe{\l}czy\'{n}ski's property $(V^\ast)$ of order $p$ and its quantification, Math. Nachrichten \textbf{291} (2018), 420--442.

\bibitem{NaB}
L. Narici, E. Beckenstein,  \emph{Topological vector spaces}, Second Edition, CRC Press, New York, 2011.

\bibitem{SaMo}
M. Salimi, S.M. Moshtaghioun, The Gelfand--Phillips property in closed subspaces of some operator spaces. Banach J. Math. Anal. \textbf{5}  (2011), 84--92.

\end{thebibliography}

\end{document}